\documentclass[11pt]{article}
\usepackage{amssymb,amsmath,amsfonts,amsthm,epsfig,latexsym,color}

\makeatletter
\newcommand{\subsectionruninhead}{\@startsection{subsection}{2}{0mm}
{-\baselineskip}{-0mm}{\bf\large}}
\newcommand{\subsubsectionruninhead}{\@startsection{subsubsection}{3}{0mm}
{-\baselineskip}{-0mm}{\bf\normalsize}}
\makeatother

\newtheorem*{theorem*}{Theorem}

\newtheorem{proposition}{Proposition}[section]
\newtheorem*{proposition*}{Proposition}
\newtheorem*{corollary*}{Corollary}

\newtheorem*{claim*}{Claim}
\theoremstyle{definition}

\newtheorem*{definition*}{Definition}
\newtheorem*{definitions*}{Definitions}
\theoremstyle{remark}

\newtheorem*{remark*}{Remark}
\newtheorem*{remarks*}{Remarks}

\numberwithin{equation}{section}

  \def\CC{{\mathbb C}}

 \def\NN{{\mathbb N}}  
 \def\RR{{\mathbb R}} \def\SS{{\mathbb S}} \def\TT{{\mathbb T}}
   
 \def\ZZ{{\mathbb Z}}

\def\cB{\mathcal{B}}    
\def\cC{\mathcal{C}}    
\def\cD{\mathcal{D}}    \def\cV{\mathcal{V}}

\newcommand{\diff}{\operatorname{Diff}}
\newcommand{\id}{\operatorname{Id}}

\newcommand{\diam}{\operatorname{Diam}}

\begin{document}
\title{A dynamical decomposition of the torus into pseudo-circles}
\author{Fran\c{c}ois B\'eguin, Sylvain Crovisier, Tobias J\"ager}
\maketitle

\begin{flushright}
\it To the memory of Dmitri V. Anosov.
\end{flushright}

\begin{abstract} We build an irrational pseudo-rotation of the
  $2$-torus which is semiconjugate to an irrational rotation of the
  circle in such a way that all the fibres of the semi-conjugacy are
  pseudo-circles. The proof uses the well-known {\em
    `fast-approximation method'} introduced by Anosov and Katok. 
\end{abstract}

\section{Introduction} It is well known that continua
(connected compact metric spaces) with complicated structure naturally
appear in smooth surface dynamics.  A striking example is provided by
the \emph{pseudo-circle}, introduced by Bing~\cite{bing} and
characterized by Fearnley~\cite{fearnley2}.  It is a continuum
which:
\begin{itemize}
\item[--] can be embedded in $S^2$ and separates,
\item[--] is circularly chainable: it admits coverings into compact subsets
  $(A_i)_{i\in \ZZ/n\ZZ}$ whose diameter are arbitrarily small, such that
  $A_i\cap A_j\neq \emptyset$ if and only if if $i=j\pm 1$ or $i=j$,
\item[--] is indecomposable:
it cannot be written as the union of two proper continua,
\item[--] and whose non-trivial proper subcontinua are
  indecomposable, homogeneous (any point can be sent on any other
  point by some homeomorphism) and all homeomorphic to the same
  topological space (called the \emph{pseudo-arc}).
\end{itemize}

Handel~\cite{handel} has built a smooth diffeomorphism of $\SS^2$
preserving a minimal invariant set homeomorphic to the
pseudo-circle.  Later, Prajs~\cite{prajs} has constructed a
partition of the annulus into pseudo-arcs, and likewise his method
could be used to produce partitions of the torus into
pseudo-circles. It was not known, however, if such a pathological
foliation could be `dynamical', that is, invariant under the
dynamics of a torus homeomorphism or diffeomorphism that permutes
the leaves of the foliation. Conversely, if a homeomorphism of the
two-torus is semiconjugate to an irrational rotation of the circle,
one may wonder whether most, or at least some, of the fibres of the
semi-conjugacy must have a simple structure or even be topological
circles. We give a positive answer to the first and a negative to
the second of these questions. Denote by $\TT^d=\RR^d/\ZZ^d$ the
$d$-dimensional torus and by $\diff^\omega_{\text{vol},0} (\TT^2)$
the space of real-analytic diffeomorphisms of $\TT^2$ that are
isotopic to the identity and preserve the canonical volume. 
\begin{theorem*} \label{t.main} There exists a minimal diffeomorphism $f\in
  \diff^\omega_{\text{vol},0} (\TT^2)$ whose rotation set is reduced to a unique
  totally irrational vector and which preserves a partition $\cC$ of $\TT^2$
  into pseudo-circles.

Moreover there exists a continuous map $p\colon \TT^2\to \TT^1$
which semi-conjugates $f$ to an irrational rotation. The elements of $\cC$
are the pre-images $p^{-1}(x)$.
\end{theorem*}

This result has implications for a number questions that naturally come up in
the rotation theory on the torus and, more specifically, the dynamics of
irrational pseudo-rotations. We discuss these issues in more detail in
Section~\ref{s.unique}, alongside with the uniqueness of the
semi-conjugacy.
\medskip

\noindent {\it Idea of the construction.}  The diffeomorphism $f$ is
obtained as limit of a sequence of diffeomorphisms $f_n$ that are
conjugated to rational rotations $R_{\alpha_n}$ by diffeomorphisms
$H_n$ isotopic to the identity, following the celebrated
Anosov-Katok method, see~\cite{fayad-katok}. As a side effect, this
means that its dynamics can be made uniquely ergodic, although we will
not expand on this. Note that most of the constructions using this
method deal with the $C^\infty$ category; some cases,
as~\cite{fayad-katok2} allow to work in the real-analytic category.
  The sequence $f_n=H_n^{-1}\circ R_{\alpha_n}\circ H_n$ is
obtained inductively. At stage $n$, the diffeomorphism $f_n$
preserves the foliation $\cV_n$ by vertical circles
$H_n^{-1}(\{x\}\times \TT^1)$. The main requirement is to have the
circles of the foliation $\cV_{n+1}$ arbitrarily close to the circles
of $\cV_n$ in the Hausdorff topology, but more crooked. Then the
partitions $\cV_n$ will converge to the partition into
pseudo-circles.  The foliation $\cV_{n+1}$ is built inductively in
the chart defined by the conjugacy $H_n$, as the preimage of the
foliation into vertical lines under a new homeomorphism $h_{n+1}$. A
sketch of the construction is given in Figure~\ref{f.strategy}.  Since
we then let $H_{n+1}=h_{n+1}\circ H_n$, this ensures that the
foliation $H_{n+1}(\cV_{n+1})$ is again the one given by vertical
circles. In order to obtain $h_{n+1}$, one first builds a leaf of
$\cV_{n+1}$ in the chart given by $\cV_n$ which is crooked with
respect to the vertical circles in such a way that the leaf is
transverse to a linear flow $\varphi_{n+1}$. This first leave is the
image of some vertical circle. One then obtains the complete
foliation (and thus the definition of $h_{n+1}$) by pushing this
initial leave by the flow. 
\begin{figure}[h!]
\vspace{-0.3cm}
\begin{center}
\includegraphics[width=0.9\linewidth]{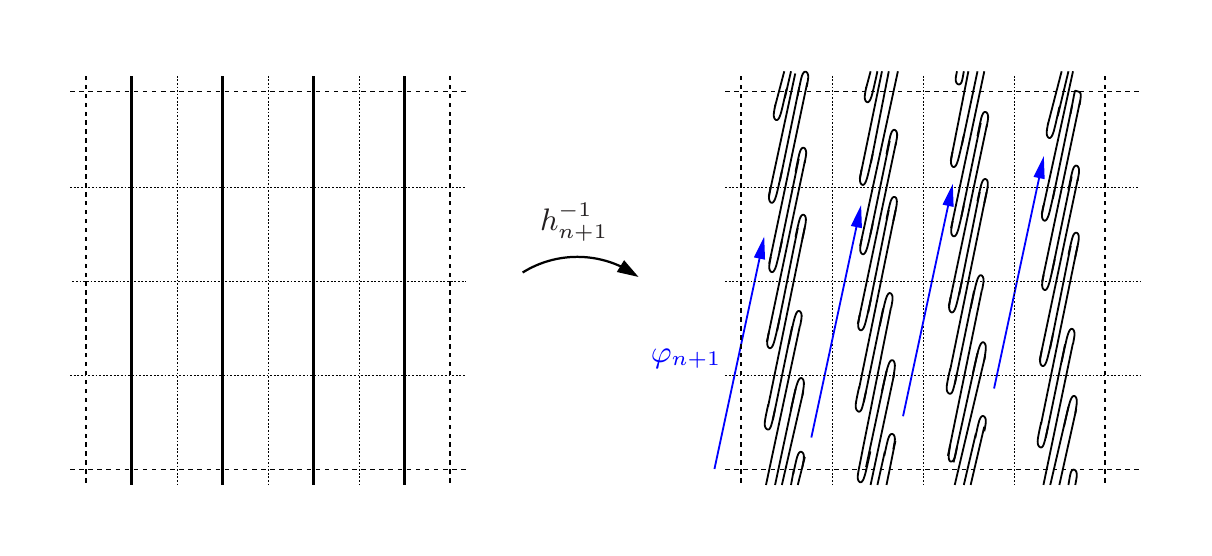}
\end{center}
\vspace{-1cm}
\caption{\it In the coordinates defined by $H_n$ (given by the dashed lines),
  $h_{n+1}^{-1}$ maps straight lines (solid lines on the left) to crooked leaves
  of $\cV_{n+1}$ (on the right). The latter are translates of each other along
  the flow lines of $\varphi_{n+1}$. \label{f.strategy}}
\end{figure}
\medskip

We note that A. Avila has announced recently the construction of an
element of $\diff^\omega_{\text{vol},0} (\TT^2)$ whose rotation set
is a non-trivial compact interval contained in line with irrational
slope and which does not contain any rational point (a
counter-example to one case of a conjecture by Franks and
Misiurewicz~\cite{franks-misiurewicz}).  His construction may
somewhat be compared to ours: the diffeomorphism is obtained as the limit
of diffeomorphisms acting periodically on the leaves of a foliation
by circles. In his case however the homotopy class of the leaves is
modified at each stage of the construction.

\section{A criterion for the existence of a partition into pseudo-circles}

\paragraph{2.a -- Crooking.}\label{Crooking}
The construction of the pseudo-circle uses the following notions.

\begin{definitions*} A \emph{circular chain} is finite family of
  sets $\cD=\{D_{\ell}, \ell\in \ZZ/N\ZZ\}$ such that $D_{k}$
  intersects $D_{\ell}$ if and only if $k-\ell\in \{-1,0,+1\}$.  A
  circular chain $\cD'=\{D'_{i}, i\in \ZZ/N'\ZZ\}$ said to be
  \emph{crooked inside} another circular chain $\cD=\{D_{\ell},
  \ell\in \ZZ/N\ZZ\}$ if there exists a map $\ell \colon \ZZ/N'\ZZ\to
  \ZZ/N\ZZ$ with the following properties.
\begin{itemize}
\item[--] $D'_{i}\subset D_{\ell(i)}$ for each $i\in \ZZ$;
\item[--] if $i<j$ are such that for all $i<k<j$ the element $\ell(k)$
  belongs to the same interval bounded by $\ell(i)$ and $\ell(j)$
  (either positively or negatively oriented in $\ZZ/N\ZZ$) and the
  length of this interval is greater than 4, then there exists $u,v$
  with $i<u<v<j$ such that $d(\ell(u),\ell(j))\leq 1$ and
  $d(\ell(v),\ell(i))\leq 1$. (Here $d$ denotes the canonical distance
  on $\ZZ/N\ZZ$.) \end{itemize}
\end{definitions*}

The pseudo-circle can then be obtained as follows. For the sake of consistency
with the later sections, we work in the torus instead of $\RR^2$ and require
that the circular chains -- and thus the resulting pseudo-circle -- are
homologically non-trivial. 

\begin{theorem*}[\cite{bing,fearnley2}] Consider a sequence
  $(\cD_n)_{n\geq 0}$ of circular chains of open topological disks in
  $\TT^2$.  Assume that
\begin{itemize}
\item[--] $\cD_{n+1}$ is crooked inside $\cD_{n}$ for each $n$,
\item[--] the closure of $\bigcup_{D\in \cD_{n+1}} D$ is contained in
  $\bigcup_{D\in \cD_n} D$ for every $n$, 
\item[--] the maximal diameter of the elements of $\cD_n$ goes to zero as $n\to
  +\infty$.
\item[--] the union $\bigcup_{D\in\cD_n} D$ contains homotopically non-trivial loops
  of a unique homotopy type $v\in\ZZ^2\setminus \{0\}$.
\end{itemize}

Then the compact set $X:=\bigcap_{n\geq 0}\bigcup_{D\in\cD_n} D$ is
homeomorphic to the pseudo-circle. Moreover, $X$ is an annular continuum of
homotopy type $v$ (see \cite{JP,JT}).
\end{theorem*}

At some point later on, we will have to speak about lifts of circular
chains in the torus to the universal covering $\RR^2$, and similarly
about lifts of circular chains of intervals in the circle to $\RR$.
Suppose that $\cD$ is a circular chain of topological disks in $\TT^2$ as above
  and denote by $\pi:\RR^2\to\TT^2$ the canonical projection. Note
  that for each $D_\ell\in\cD$, the preimage $\pi^{-1}(\cD)$ consists
  of a countable number of connected components, each of which is a
  topological disk homeomorphic to $D_\ell$ and disjoint from all its
  integer translates. Suppose in addition that $\diam(D_\ell)<1/4$ for
  all $\ell\in\ZZ/N\ZZ$, so that none of the unions $D_\ell\cup
  D_{\ell+1}$ is essential in the torus (contains a homotopically
  non-trivial loop).

\begin{definitions*}
A \emph{lift} $\widehat\cD$ of $\cD$ is a sequence of topological disks $(\widehat D_\ell)_{\ell\in\ZZ}$ of $\RR^2$
  such that
\begin{itemize}
\item[--] for all $\ell\in\ZZ$ the disk $\widehat D_\ell$ is a
  connected component of $\pi^{-1}(D_\ell)$;
\item[--] $\widehat D_k$ intersects $\widehat D_\ell$ if and only if
  $k-\ell\in\{-1,0,1\}$.
\end{itemize}
The disc $\widehat D_{k+N}$ is the image of $\widehat D_{k}$ by translation by a vector
$v\in \ZZ^2$ which does not depend on the lift, nor on $k$, and is called the \emph{homotopy type of $\cD$}.
\end{definitions*}
Note that if $\cD$ and $\cD'$ are circular chains of topological disks with
homotopy type $v\in\ZZ^2\setminus\{0\}$ as above, $\cD'$ is crooked inside $\cD$
and $\widehat \cD,\widehat \cD'$ are lifts in the above sense, then there exists
a function $\hat \ell:\ZZ\to\ZZ$ (to which we refer as a lift of
$\ell:\ZZ/N'\ZZ\to\ZZ/N\ZZ$) such that
\begin{itemize}
\item[--] $\widehat D'_{i+N'}=\widehat D'_i+v$ and $\widehat D_{i+N}=\widehat
  D_i+v$;
\item[--] $\hat\ell(i+N')=\hat\ell(i)+N$;
\item[--] $\widehat D'_i\subseteq \widehat D_{\ell(i)}$ for all $i\in\ZZ$;
\item[--] if $i<j<i+N$ are such that for all $i<k<j$ the integer
  $\hat\ell(k)$ belongs to the interval bounded by $\hat\ell(i)$ and
  $\hat\ell(j)$ and $|\hat\ell(j)-\hat\ell(i)|>4$, then there exists
  $u,v$ with $i<u<v<j$ such that $|\hat\ell(u)-\hat\ell(j)|\leq 1$ and
  $|\hat\ell(v),\hat\ell(i)|\leq 1$.
\end{itemize}
All these remarks apply in an analogous way to circular chains of
intervals in the circle and their lifts to $\RR$. \medskip

During the construction, we will also use another notion of the crooking.

\begin{definition*} For $\varepsilon>0$, a continuous map $g\colon
  I\to \RR$ on the interval $I$ is \emph{$\varepsilon$-crooked} if
  for any $a<b$ in $I$, there are $a<c<d<b$ such that
  $|g(d)-g(a)|<\varepsilon$ and $|g(c)-g(b)|<\varepsilon$.
\end{definition*}

Note that $\varepsilon$-crooked maps exist for any $\varepsilon$ (see~\cite{bing2}).

\medskip

\paragraph{2.b -- Elements of the construction.} Let $\pi\colon
\TT^2\to \TT^1$ be the projection on the first coordinate. Let
$\cB(N)$ be the covering of the circle by $N$ open intervals defined
as follows
\begin{equation}\label{e.defB}
  \cB(N)=\{B_i, i\in \ZZ/N\ZZ\} \quad \mbox{where}\quad
  B_i=\left(\frac {i-5/4} {N},\frac{i+1/4}{N}\right).
\end{equation}
We will build inductively:
\begin{itemize}
\item[--] a sequence of integers $(N_n)_{n\geq 0}$,
\item[--] a sequence of positive real numbers $(\varepsilon_n)_{n\geq 0}$,
\item[--] a sequence of conjugating diffeomorphisms $(H_n)_{n\geq 0}$
  in $\diff^\omega_{\text{vol},0}(\TT^2)$,
\item[--] a sequence of rational rotations $(R_{\alpha_n})_{n\geq 0}$ of $\TT^2$.
\end{itemize} To $N_n,\varepsilon_n,H_n,R_{\alpha_n}$, we will
associate:
\begin{itemize}
\item[--] for each $x\in\TT^1$, the annulus $A_{n,x}$ which is the
  image under $H_n^{-1}$ of the vertical annulus
  $(x-\varepsilon_n,x+\varepsilon_n)\times \TT^1$,
\item[--] for each $x\in\TT^1$, the covering $\cD_{n,x}$ of the
  annulus $A_{n,x}$ defined by 
$$\cD_{n,x} = \left\{H_n^{-1}((x-\varepsilon_n,x+\varepsilon_n)\times B_i)\ \mid \
   B_i\in\cB(N_n)\right\}$$
(note that $\cD_{n,x}$ is a circular chain with $N_n$ elements),
\item[--] the projection $p_n=\pi_1\circ H_{n}: \TT^2\to\TT^1$ (where
  $\pi_1:\TT^2\to\TT^1$ is the projection to the first coordinate),
\item[--] the diffeomorphism $f_n=H_n^{-1}\circ R_{\alpha_n}\circ
  H_n\in \diff^\omega_{\text{vol},0}\left(\TT^2\right)$.
\end{itemize} 
We will denote by $\left(\frac{r_n}{q_n},\frac{s_n}{q_n}\right)$ the coordinates of
$\alpha_n$, with $r_n,s_n\in\mathbb{Z}$ and $q_n\in\NN-\{0\}$.

\paragraph{2.c -- Inductive properties.} The torus $\TT^2$ is
embedded as the subset $\{(z_1,z_2)\in \CC^2,\; |z_1|=|z_2|=1\}$ of
$\CC ^2$. One will consider homeomorphisms $f$ of $\TT^2$ such that
both $f$ and $f^{-1}$ extend as holomorphic functions defined on a
neighborhood of $\Delta=\left\{(z_1,z_2)\in \CC^2,\; |z_1|,|z_2|\in\left[\frac 1
  2, 2\right]\right\}$. One then introduces the supremum norm $\|.\|_\Delta$ on
$\Delta$ and the metric
$$ d_0(f,f')=\max(\|f,f'\|_\Delta,  \|f^{-1},{f'}^{-1}\|_\Delta)\ .$$
 The sequences $(N_n)$, $(\varepsilon_n)$, $(H_n)$,
$(R_{\alpha_n})$ will be constructed inductively so that the
following properties hold.
\begin{enumerate}
\item\label{a1} For each $x\in \TT^1$, the circular chain $\cD_{n+1,x}$ is
  crooked inside the circular chain $\cD_{n,x}$ (in particular, the annulus
  $A_{n+1,x}$ is contained in the interior of the annulus $A_{n,x}$), and the
  supremum of the diameter of the elements of the coverings $\cD_{n+1,x}$ is
  less $\frac{1}{n+1}$,
\item\label{a2} The angle $\alpha_{n+1}$ is close, but not
  equal, to $\alpha_n$. More precisely:
  $|\frac{r_{n+1}}{q_{n+1}}-\frac{r_n}{q_n}|$ and
  $|\frac{s_{n+1}}{q_{n+1}}-\frac{s_n}{q_n}|$ are smaller than $1/(2^{n+1}q_n)$ and
  the orbits of $R_{\alpha_{n+1}}$ are $\frac{1}{2^{n+1}}$-dense in $\TT^2$.
\item\label{a25} Every orbit of the diffeomorphism $f_{n+1}$ is
  $\frac{1}{n+1}$-dense in $\TT^2$.
\item\label{a3} The diffeomorphism $f_{n+1}$ is (very) close to $f_n$. More precisely,
both $f_{n+1}$, $H_{n+1}$ and their inverses
extend holomorphically on $(\CC\setminus \{0\})^2$ and satisfy:
\begin{enumerate}
\item $d_{0}(f_{n+1}^i,f_n^i)<\min (\frac{1}{2}d_{0}(f_{n}^i,f_{n-1}^i), \frac{1}{n})$ for $i=1,\dots,q_{n}$;
\item $d_{0}(f_{n+1},f_n)<\frac{\eta_n}{2}$ where $\eta_n$ is choosen such that,
  for every homeomorphism $g$ in the ball (for $d_0$) centered at $f_n$ of
  radius $\eta_n$, the rotation set of $g$ is contained in the ball centered at
  $\alpha_n$ of radius $\frac{1}{n}$.
\end{enumerate}
\item\label{a4} The projection $p_{n+1}$ is close to $p_{n}$ for the
  $C^0$-topology. More precisely: $d_{0}(p_{n+1},p_n)<\frac{1}{2^n}$.
\end{enumerate}

\begin{remarks*} The existence of the real number $\eta_n$ used in
  property~\ref{a4}.b is a consequence of the upper semi-continuity
  of the rotation set $\rho(F)$ with respect to $F$
  \cite[Corollary 3.7]{MZ}.

  Property~\ref{a1} (more precisely, the fact that $\cD_{n+1,x}$ is crooked
  inside $\cD_{n,x}$) implies that the sequence of conjugating diffeomorphisms
  $(H_n)$ will necessarily diverge. Nevertheless, the sequence of
  diffeomorphisms $(f_n)=(H_n^{-1}\circ R_{\alpha_n}\circ H_n)$ will converge
  (Property~\ref{a3}). This convergence is obtained by using the well-known
  ingredients of the Anosov-Katok method:
\begin{itemize}
\item[--] one first chooses a conjugating diffeomorphism $H_{n+1}$ of the form
  $H_{n+1}=h_{n+1}\circ H_n$, where $h_{n+1}$ might be very wild, but commutes
  with the rotation $R_{\alpha_n}$ ; this implies that $f_n=H_{n+1}^{-1}\circ
R_{\alpha_n}\circ H_{n+1}$ ;
\item[--] then, choosing $\alpha_{n+1}$ close enough to $\alpha_n$ is enough to
  ensure that $f_{n+1}=H_{n+1}^{-1}\circ R_{\alpha_{n+1}}\circ H_{n+1}$ is close
  to $f_{n}=H_{n+1}^{-1}\circ R_{\alpha_n}\circ H_{n+1}$.
\end{itemize}

A specific point in our construction is that, although the sequence
of diffeomorphisms $(H_n)$ will diverge, we require that the
sequence of maps $(\pi\circ H_n)$ converges (Property~\ref{a4}).
Indeed, we want that the fibers of $\pi_1\circ H_n$ converge to
pseudo-circles ``foliating'' $\TT^2$.
\end{remarks*}

\paragraph{2.d -- Proof of the theorem.} One can easily check that
the theorem follows from the inductive properties stated above. 
Properties~\ref{a1} imply that, for every $x\in\TT^1$, the sequence of
annuli $(A_{n,x})$ decreases and converges in Hausdorff topology to
a pseudo-circle $\cC_x$. Moreover, the collection of pseudo-circles
$\cC=\{\cC_x\}_{x\in\TT^1}$ is a partition of $\TT^2$; this follows
from the following fact:
\begin{itemize}
\item[--] for every $n$, the collection of annuli $\{A_{n,x}, x\in\TT^1\}$ covers $\TT^2$,
\item[--] for any $x\neq x'$, the annuli $A_{n,x}$ and $A_{n,x'}$ are
  disjoint if $n$ is large enough.
\end{itemize}

Property~\ref{a3}.a implies that the sequence $(f_n)$ converges to
an holomorphic function $f$ on the interior of $\Delta$. The same
holds for $(f_n^{-1})$. Consequently, the restriction of $f$ to
$\TT^2$ is a real-analytic diffeomorphism.  Since each $f_n$ is
volume-preserving, $f$ belongs to
$\diff^\omega_{\text{vol},0}(\TT^2)$.

Given $x\in\TT^1$, let $x':=x+\pi(\alpha_n)$. Then $f_n(A_{n,x})=A_{n,x'}$. From
property~\ref{a1}, one deduces that $A_{n+1,x}$ is mapped by $f_n$ inside
$A_{n,x'}$ and $A_{n+1,x'}$ is mapped by $f_n^{-1}$ inside $A_{n,x}$. Hence, the
annulus $f(A_{n+1,x})$ is contained in the $d_{0}(f,f_n)$-neighbourhood of the
annulus $A_{n,x'}$ and the the annulus $f^{-1}(A_{n+1,x'})$ is contained in the
$d_{0}(f,f_n)$-neighbourhood of the annulus $A_{n,x}$. Since $d_{0}(f,f_n)$
tends to $0$ as $n$ goes to infinity, this implies that $f$ preserves the
partition in pseudo-circles $\cC=\{\cC_x\}_{x\in\TT^1}$.

Property~\ref{a2} implies that the sequence $(\alpha_n)$ converges towards some
$\alpha\in\RR^2$. It also implies that the $q_n$ first iterates of $R_\alpha$ are $1/2^{n+1}$
close to those of $R_{\alpha_{n+1}}$, hence are $\frac{1}{2^n}$-dense in $\TT^2$.
Consequently $\alpha$ is totally irrational.

Properties~\ref{a3} imply that the rotation set of $f$ is
reduced to $\{\alpha\}$ (indeed, they imply that
$d_{0}(f,f_n)<\eta_n$ and therefore the rotation set of $f$ is contained in the
ball of radius $\frac{1}{n}$ centered at $\alpha_n$ for every $n$). Hence $f$ is
an irrational pseudo-rotation.

Consider a point $z\in\TT^2$. For every $n$, according to
property~\ref{a25}, the orbit of $z$ under $f_n$ is
$\frac{1}{n}$-dense in $\TT^2$. But the orbit of $z$ under $f_n$ is
periodic of period less than $q_n$ (since $f_n$ is conjugate to the
rotation $R_{\alpha_n}$). Using property~\ref{a3}.a, we obtain that
the orbit of $z$ under $f$ remains at distance less than
$\frac{2}{n}$ of the orbit of $z$ under $f_n$ for a time $q_n$.
Combined with property 3, this means that the orbit of $z$ under $f$
is $\frac{3}{n}$-dense in $\TT^2$. Since $n$ is arbitrary, $f$ is
minimal.

Property~\ref{a4} implies that the sequence of maps $(p_n)$ converges in
topology $C^0$ towards a continuous map $p$. For each $n$, the map $p_n$
semi-conjugates $f_n$ to the rotation of $\TT^1$ with angle
$\pi(\alpha_n)$. Passing to the limit, it follows that the map $p$
semi-conjugates $f$ to the rotation of angle $\pi(\alpha)$. So we get all the
conclusions of the theorem.  \qed

\section{Inductive construction} Now we explain how to construct a
sequence of integers $(N_n)_{n\geq 0}$, a sequence of real numbers
$(\varepsilon)_{n\geq 0}$, a sequence of conjugating diffeomorphisms
$(H_n)_{n\geq 0}$ and a sequence of vectors $(\alpha_n)_{n\geq 0}$, so
that properties 1\ldots 5 are satisfied. We assume that the sequences
have already been constructed up to rank $n$. We will construct
$N_{n+1},\varepsilon_{n+1},H_{n+1},\alpha_{n+1}$.

\paragraph{3.a -- Preliminary constructions.} Recall that we denote by
$\left(\frac{p_n}{q_n},\frac{r_n}{q_n}\right)$ the coordinates of
$\alpha_n$. We introduce a periodic linear flow
$\varphi_{n+1}:(t,(x,y))\mapsto
(x,y)+t\cdot\alpha_n+t\cdot(0,p_nb_{n+1})$ on $\TT^2$ where $b_{n+1}$
is an integer which will be specified below. Observe that the time 1
map of this flow is the rotation $R_{\alpha_n}$. The first return map
of $\varphi_{n+1}$ on the vertical circle $\{x\}\times\TT^1$ is the
time $\frac{q_n}{p_n}$ map of $\varphi_{n+1}$. Denote by $m_n$ the
period of this first return map, and observe that $m_n$ depends on
$\alpha_n$, but does not depend on the choice of the integer
$b_{n+1}$.

We also introduce a $C^\omega$ map $\theta_n\colon \RR\to \RR$ satisfying the following properties.
\begin{itemize}
\item[--] $x\mapsto \theta_n(m_nx)-m_nx$ is a trigonometric polynomial
  (hence $\theta_n-\id$ is $1/m_n$-periodic),
\item[--] $\theta_n(0)=0$, $\theta_n\left(\frac 1{2m_n}\right)=2$,
  $\theta_n\left(\left[0,\frac 1{2m_n}\right]\right)=[0,2]$ and $\theta_n$ is
  $\frac{1}{4N_n}$-crooked on $\left[0,\frac 1{2m_n}\right]$,
\item[--] $\theta_n\left(\frac 1{m_n}\right)=\frac 1{m_n}$,
  $\theta_n\left(\left[\frac 1{2m_n},\frac
      1{m_n}\right]\right)=\left[\frac 1{m_n},2\right]$ and $\theta_n$
  is $\frac{1}{4N_n}$-crooked on $\left[\frac 1{2m_n},\frac 1{m_n}\right]$.
\end{itemize}

Note that $\theta_n-\id$ induces a function on $\TT^1$ which extends
holomorphically to $\CC\setminus \{0\}$. \medskip  Recall that $\cB(N)$
denotes the covering of the circle by $N$ compact intervals defined
by~\eqref{e.defB}. Moreover, $\theta$ itself induces a degree one map on the
circle, which we denote by $\theta$ again for simplicity.

\begin{claim*} If $N_{n+1}$ is large enough, then, for any
  $\omega\in \TT^1$, the circular chain of intervals
  $\{\theta_{n}(B-\omega)+\omega, B\in \cB(N_{n+1})\}$ is crooked
  inside the circular chain $\cB(N_{n})$.
\end{claim*}

\begin{proof}

  We work with a lift of the family $\cB(N)$, in the sense discussed in
  Section~\ref{Crooking}, obtained as a covering $\widehat \cB(N)$ of the real
  line by intervals of the form:

$$\widehat B_{i}=\bigg(\frac {i-5/4} {N},\frac{i+1/4}{N}\bigg), \quad i\in \ZZ.$$
 If $N_{n+1}$ is large enough, each interval $\theta_{n}(\widehat
B_i-\omega)+\omega$ with $\widehat B_i\in \widehat \cB(N_{n+1})$ has
lenght strictly less than $\frac{1}{2N_n}$ and therefore is contained
in an interval $\widehat B'_{\hat\ell(i)}$ of the family $\widehat
\cB(N_n)$. Since $\theta_n$ has degree $1$, one can choose the
function $\hat\ell$ such that
$\hat\ell(i+N_{n+1})=\hat\ell(i)+N_{n}$.

Let $i<j$ be two integers such that $\hat\ell(k)$ belongs to the
interval bounded by $\hat\ell(i)$ and $\hat\ell(j)$ for each $i<k<j$
and such that $4<|\hat\ell(j)-\hat\ell(i)|<N_n$. Let us choose $i\leq
i'<j'\leq j$ such that $\hat\ell(i')=\hat\ell(i)$,
$\hat\ell(j')=\hat\ell(j)$ and some points $a\in\widehat B_{i'}$ and
$b\in \widehat B_{j'}$.  One considers $a\leq a'<b'\leq b$ such
that $a,a'$ (resp. $b,b'$) have the same image by $x\mapsto
\theta_{n}(x-\omega)$. One can assume that $a'-\omega,b'-\omega$
belong to the same interval $I_k:=\left[\frac {k}{2m_n},\frac
  {k+1}{2m_n}\right]$: indeed, for each $k$, the image under
$\theta_n$ of $I_k$ has length $>1$ and the point
$\theta_n\left(\frac {k}{2m_n}\right)$ is an end point of
$\theta_n(I_{k-1}\cup I_k)$.

Now the points $a'-\omega$ and $b'-\omega$ belong to the same
interval $I_k=\left[\frac {k}{2m_n},\frac {k+1}{2m_n}\right]$. Since $\theta_{n}$
is $(4N_n)^{-1}$-crooked on this interval, this implies that there exists $a'<c<d<b'$ such that
\begin{itemize}
\item[--] $\theta_{n}(a-\omega)=\theta_{n}(a'-\omega)$ and $\theta_{n}(d-\omega)$ are $(4N_n)^{-1}$-close,
\item[--] $\theta_{n}(b-\omega)=\theta_{n}(b'-\omega)$ and $\theta_{n}(c-\omega)$ are $(4N_n)^{-1}$-close.
\end{itemize} Hence $d$ is contained in an interval $\widehat B_v$
such that $|\hat\ell(v)-\hat\ell(i)|\leq 1$. Similarly, $c$ is
contained in an interval $\widehat B_u$ such that
$|\hat\ell(u)-\hat\ell(i)|\leq 1$. Since $a<c<d<b$ and
$2<|\hat\ell(j)-\hat\ell(i)|$ one gets $i<u<v<j$.

In the projection to $\TT^1$, the above shows that the circular chain
$\{\theta_{n}(B-\omega)+\omega, B\in \cB(N_{n+1})\}$ is crooked inside
the circular chain $\cB(N_n)$ as announced.
\end{proof}

Since $m_n$ only depends on $\alpha_n$, one can fix the map $\theta_{n}$ but choose $b_{n+1}$ and $N_{n+1}$ arbitrarily large later
in the construction.

\paragraph{3.b -- Construction of $H_{n+1}$.} Consider the map $\Theta_{n+1}:\{0\}\times\TT^1\to\TT^1$ defined by $\Theta_{n+1}:(0,y)\mapsto \theta_n(y)-y$. 
Since $\theta_{n}-\id$ is $1/m_n$-periodic, and since the period of the return map associated to the linear flow $\varphi_{n+1}$
on $\{0\}\times \TT^1$ is equal to $m_n$, this maps extends to a map $\Theta_{n+1}:\TT^2\to \TT^1$ which is constant along the orbits of $\varphi_{n+1}$.
We define $H_{n+1}$ by setting $H_{n+1}:=h_{n+1}\circ H_n$, where
$$
h_{n+1}(x,y)=\varphi_{n+1}\left(-\frac{\Theta_{n+1}(x,y)}{p_nb_{n+1}},(x,y)\right)=\bigg(x,y-\Theta_{n+1}(x,y)\bigg)-\frac{\Theta_{n+1}(x,y)}{q_nb_{n+1}}\left(1,
\frac{r_n}{p_n}\right).
$$
Since $\Theta_{n+1}$ is constant along the orbits of $\varphi_{n+1}$, one has 
\begin{equation}\label{e.formulah}
h_{n+1}^{-1}(x,y)=\bigg(x,y+\Theta_{n+1}(x,y)\bigg)+\frac{\Theta_{n+1}(x,y)}{q_nb_{n+1}}\left(1,
\frac{r_n}{p_n}\right).
\end{equation}
Clearly, $h_{n+1}$ (and hence $H_{n+1}$) belongs to
$\diff^\omega_{\text{vol},0}(\TT^2)$. Note also that both $h_{n+1}$
and $h_{n+1}^{-1}$ extend holomorphically to the domain $(\CC\setminus
\{0\})^2$.

\begin{figure}[h!]
\vspace{-0.3cm}
\begin{center}
\includegraphics[width=0.9\linewidth]{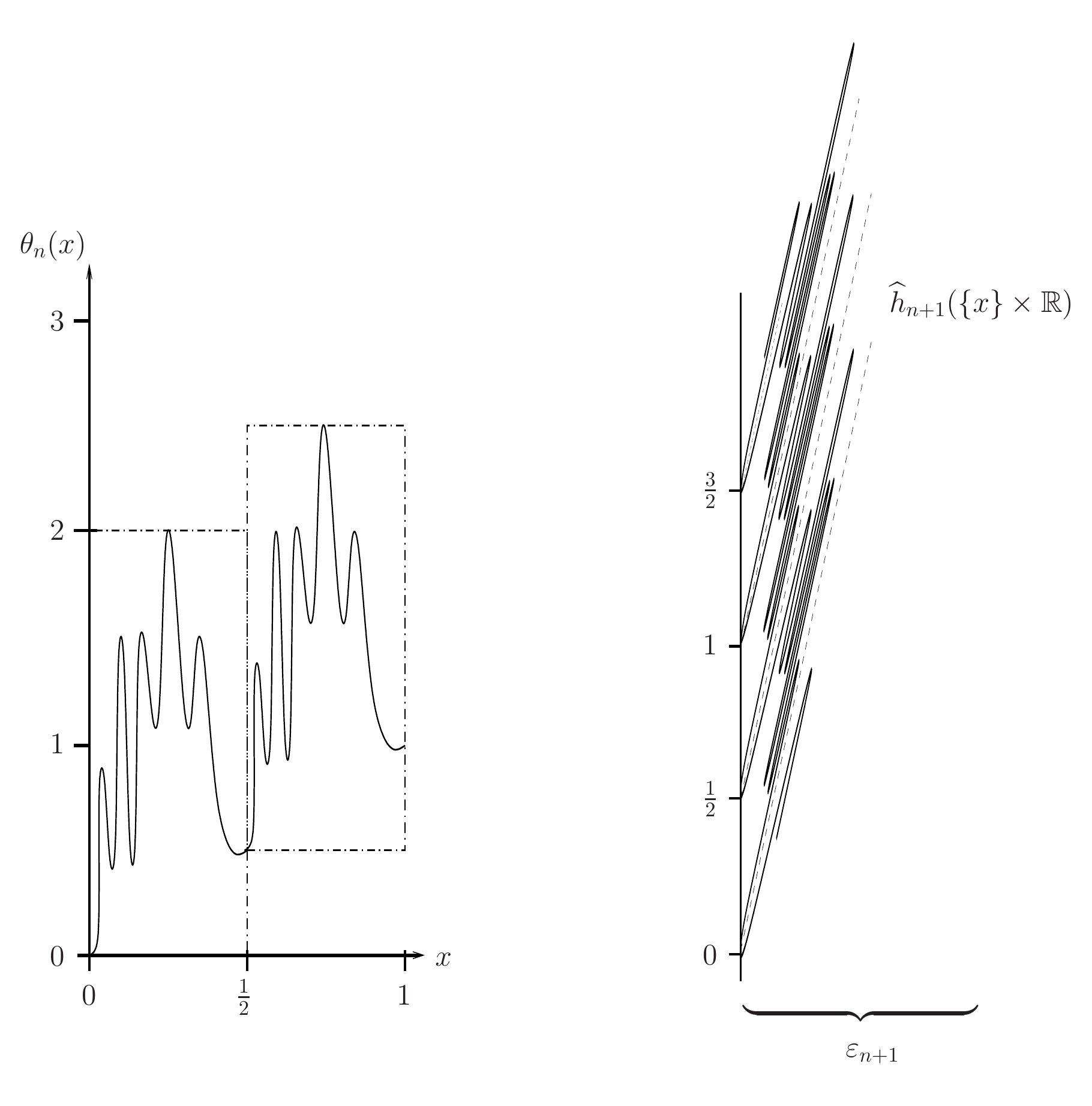}
\end{center}
\vspace{-1cm}
\caption{\it Choice of the function $\theta$ (on the right) and the image of a
  vertical line $\{x\}\times \mathbb{R}$ under the lift $\widehat h_{n+1}$ of
  $h_{n+1}$ (schematic picuture with $m_n=2$). Note that the preimages of vertical lines under $h_{n+1}$ have a
  similar {\em `crookedness'}, which is the fact that is needed for our
  construction (see  next section). \label{f.theta} The horizontal size of
  images (and preimages) of vertical lines under $h_{n+1}$ is small compared to
  $\varepsilon_{n+1}$. This size is controlled by the flow lines of $\varphi_{n+1}$ (dashed lines),
  whose direction is given by the almost vertical vector
  $\left(\frac{1}{b_{n+1}q_n+(r_n/p_n)},1\right)$. }
\end{figure}

\paragraph{3.c -- Choice of $b_{n+1}$, $\varepsilon_{n+1}$ and
  $N_{n+1}$.} 

Now, we explain how to fix the values of $b_{n+1}$,
$\varepsilon_{n+1}$, and $N_{n+1}$  so that properties~\ref{a1}
and~\ref{a4} hold.  From~\eqref{e.formulah}, if $b_{n+1}$ is large
enough, the image under $h_{n+1}^{-1}$ of every point $(x,y)\in\TT^2$
is arbitrarily close to $\left(x,y+\Theta_{n+1}(x,y)\right)$. Now
observe that, for each $x\in \TT^1$, there exists $\omega_x\in \TT^1$
such that
\begin{equation*}
\Theta_{n+1}(x,y)=\theta_{n}(y-\omega_x)-y+\omega_x.
\end{equation*}
As a consequence, if $b_{n+1}$ is large enough, the image under
$h_{n+1}^{-1}$ of every point $z=(x,y)\in\TT^2$ is arbitrarily close
to $(x,\theta_{n}(y-\omega_x)+\omega_x)$. Hence, if $b_{n+1}$ is large
enough and $\varepsilon_{n+1}$ is small enough, the image under
$h_{n+1}^{-1}$ of a rectangle
$(x-\varepsilon_{n+1},x+\varepsilon_{n+1})\times B$ is contained in an
arbitrary small neighbourhood of the square $\{x\}\times
(\theta_{n}(B-\omega_x)+\omega_x)$. Using the claim above, this
implies that the family $\cD_{n+1,x}$ is crooked inside $\cD_{n,x}$,
provided $N_{n+1}$ is large enough. By continuity of $H_{n+1}^{-1}$,
the diameter of the elements of the covering $\cD_{n+1,x}$ is less
than $\frac{1}{n+1}$ if $b_{n+1},N_{n+1}$ are large enough and
$\varepsilon_{n+1}$ is small enough. We have thus checked that
property~\ref{a1} holds, provided that $b_{n+1},N_{n+1}$ are chosen
large enough and $\varepsilon_{n+1}$ is chosen small enough.

By definition of $p_n$, $p_{n+1}$, and $H_{n+1}$, in order to check
that property~\ref{a4} is satisfied, it is enough to check that
$\pi\circ h_{n+1}$ is close to $\pi$ for the $C^0$-topology. This is a
direct consequence of the definition of $h_{n+1}$, provided that
$b_{n+1}$ is chosen large enough.

\paragraph{3.d -- Choice of $\alpha_{n+1}$.}
Clearly, one can choose $\alpha_{n+1}$ arbitrarily close to $\alpha_n$
so that property~\ref{a2} holds. By uniform continuity of $H_{n+1}$
and $H_{n+1}^{-1}$, there exists $\eta$ so that the orbits of
$f_{n+1}=H_{n+1}\circ R_{\alpha_{n+1}}\circ H_{n+1}^{-1}$ are
$\frac{1}{n+1}$-dense in $\TT^2$ provided that the orbits of the
rotation $R_{\alpha_{n+1}}$ are $\eta$-dense in $\TT^2$. One can thus
choose $\alpha_{n+1}$ arbitrarily close to $\alpha_n$ so that
property~\ref{a25} is satisfied.   By construction both $f_{n+1}$
and $f_{n+1}^{-1}$ extend holomorphically to $(\CC\setminus
\{0\})^2$. Moreover the diffeomorphism $h_{n+1}$ commutes with the
flow $\varphi_{n+1}$, hence with the rotation $R_{\alpha_n}$. 
Consequently $f_n=H_{n+1}^{-1}\circ R_{\alpha_n} \circ H_{n+1}$. This
shows that $f_{n+1}$ is arbitrarily close to $f_n$ when $\alpha_{n+1}$
is chosen arbitrarily close to $\alpha_n$. In particular,
property~\ref{a3} holds provided that $\alpha_{n+1}$ is chosen close
enough to $\alpha_n$.

\section{Uniqueness of the semi-conjugacy, non-existence of wandering
  curves and further remarks}\label{s.unique}

The aim of this last section is to discuss, somewhat informally, the
implications of our construction for some  questions arising in the
context of dynamics and rotation theory on the two-torus. Throughout this
section, we assume $f$ is an irrational pseudo-rotation with an invariant
foliation of pseudo-circles, consisting of the fibres of a semi-conjugacy
$p:\TT^2\to\TT^1$ to an irrational rotation $R_\alpha$. Moreover, we will freely
add further assumptions on $f$ if these can easily be ensured in the preceeding
Anosov-Katok-construction.
We first note that the semi-conjugacy in
Theorem~\ref{t.main} is unique, modulo post-composition by rotations.

\begin{proposition}
  The semi-conjugacy $p$ in Theorem~\ref{t.main} is uniquely determined: any
  continuous map $p'$ which is homotopic to $p$ and semi-conjugates $f$ to the
  same circle rotation as $p$ can be written as $p'=R\circ p$ where $R$ is a
  rotation of the circle.
\end{proposition}

This follows directly by combining \cite[Corollary 4.3]{JP} (uniqueness of the
semi-conjugacy provided the non-wandering set is externally transitive) with
\cite[Theorem A]{Po} (external transitivity of the non-wandering set of
irrational pseudo-rotations).
\medskip

The uniqueness of the semi-conjugacy further allows to see that $f$
does not admit any loop which is wandering (i.e. disjoint from all
its iterates) and has the same homotopy type as the pseudo-circles
(that is, homotopy vector $v_2=(0,1)$ in our construction). The
reason for this is the fact that the existence of such a loop
$\Gamma$ would allow to construct a semi-conjugacy $\tilde p$ to the
rotation $R_\alpha$ such that $\Gamma$ is contained in a single
fibre of the semi-conjugacy. Details of this construction can be
found in \cite[Lemma 3.2]{JP} (the fact that $\Gamma$ is contained in
a single fibre is not mentioned explicitly, but is obvious from the
proof). Due to the uniqueness of the semi-conjugacy (modulo
rotations) and the fact that none of the pseudo-circles of the
foliation contains any non-degenerate curves, this yields a
contradiction.

More generally, it is even possible to show that $f$ does not admit
any loops disjoint of all its iterates, regardless of the homotopy
type. This is slightly more subtle, and we only sketch the argument.
The crucial observation is the fact that we may construct $f$ such
that 
\begin{equation}
  \label{e.deviations}
  \sup \left|\left\langle F^n(z)-z-n\rho,v\right\rangle\right| \ < \ 
\infty \quad \textrm{iff} \quad v=(1,0) \ ,
\end{equation}
where $F: \RR^2\to\RR^2$ is a lift of $f$ and $\rho\in\RR^2$ is the
corresponding rotation vector. Now, if there exists a wandering loop of homotopy
type $w\in\ZZ^2\setminus\{0\}$, then it is not hard to see that
\[
\sup \left|\left\langle F^n(z)-z-n\rho,w^\perp\right\rangle\right| \ < \ \infty \ .
\]
However, according to (\ref{e.deviations}) this is only possible if $w=(0,1)$,
and this is exactly the homotopy type of the pseudo-circles which was excluded
before. Homotopically trivial wandering loops cannot exist by minimality, and
therefore no loop of any homotopy type can be wandering.

Roughly speaking, in order to prove (\ref{e.deviations}) one has to use the fact
that since the leaves of the foliations $\cV_k$ are increasingly crooked,
connected fundamental domains of these circles in the lift become arbitrarily
large in diameter. Since an iterate of $f_k$ acts as a rotation on these leaves,
this allows to see that for suitable integers $n_k$ the vertical deviations
$\left|\left\langle F_k^{n_k}(z)-z-n_k\rho,w^\perp\right\rangle\right|$ become
arbitrarily large. If the $f_k$ converge to $f$ sufficiently fast, then this
carries over to the limit and yields unbounded vertical deviations for $f$. At
the same time horizontal deviations (that is, $v=v_1=(1,0)$ in
(\ref{e.deviations})) are bounded due to the existence of the semi-conjugacy
(e.g. \cite[Lemma 3.1]{JT}). Together, these two facts yield unbounded
deviations for all $v\neq v_1$.
\medskip

The fact that $f$ does not admit any wanding loops is of some interest
in the context of the Arc Translation Theorem due to Kwapisz
\cite{Kwa,BCL}, which asserts that given an irrational
pseudo-rotation and any integer $n$, there exist essential loops
which are disjoint from their first $n$ iterates. It is natural to
ask under what additional assumptions this statement can be
strengthened by passing from a finite number to all iterates. A
natural obstruction is certainly to have unbounded deviations in all
directions, as discussed above. However, our example shows that even
if the deviations are bounded in some direction, the existence of a
wandering curve is not guaranteed. Thus, in this sense the statement
of the Arc Translation Theorem is optimal, and essential loops have
to be replaced by more general classes of essential continua in order
to obtain results in infinite time.  \medskip

Finally, we want to mention a loose connection of our construction to the
Franks-Misiurewicz Conjecture \cite{franks-misiurewicz}. The latter asserts 
that if the rotation set of a
torus homeomorphism is a line segment of positive length, then either it
contains infinitely many rational points, or it has a rational endpoint. As
mentioned in the introduction, Avila has recently announced a counterexample to
this conjecture for the case where the rotation segment has irrational slope,
but the line it defines does not pass through a rational point. Conversely, Le
Calvez and Tal have announced the first positive partial result on the
conjecture: if the rotation set is a segment with irrational slope, it cannot
contain a rational point in its relative interior. One case which is still
completely open, however, is whether the rotation set can be a line segment with
rational slope, but without rational points -- for example, of the form
$\{\alpha\}\times [a,b]$ with $\alpha\in\RR$ irrational and $a<b$.  Now, in the
situation where there exists a semi-conjugacy $p$, homotopic to $\pi_1$, to the
rotation $R_\alpha$ on the circle, the rotation set has to be contained in the
line $\{\alpha\}\times\RR$ \cite[Lemma 3.1]{JT}. Hence, this is a natural class
of maps to look for counterexamples to this subcase of the Franks-Misiurewicz
Conjecture. However, it is known that in order to have a non-degenerate rotation
interval, the fibres of the semi-conjugacy need to have a complicated structure
-- more precisely, they need to be indecomposable \cite{JP}. Our example shows
that such a rich fibre structure is possible in principle. Yet, whether a
non-degenerate rotation interval can be achieved remains open. Here, the fact
that the Anosov-Katok method typically leads to uniquely ergodic examples
suggests that a different approach would be needed to produce such examples, if
these exists at all.

\bigskip

\noindent
\emph{Fran\c cois B\'eguin},
{\small LAGA,  CNRS - UMR 7539, Universit\'e Paris 13, 93430 Villetaneuse, France.}\\

\noindent
\emph{Sylvain Crovisier}, {\small LMO, CNRS - UMR 8628, Universit\'e Paris-Sud 11, 91405 Orsay, France.}\\

\noindent
\emph{Tobias J\"ager}, {\small Institute of Mathematics,
  Friedrich-Schiller-University Jena, 07743 Jena, Germany.}

\end{document}